\documentclass[11pt]{article}

\usepackage{todonotes}

\usepackage[top=3.5cm,bottom=3.5cm,left=3.5cm,right=3cm]{geometry}
\pdfpagewidth\paperwidth
\pdfpageheight\paperheight
\usepackage[english,italian]{babel}
\usepackage[T1]{fontenc}
\usepackage[latin1]{inputenc} 
\usepackage{lmodern}
\usepackage{amsfonts}
\usepackage{cancel}
\usepackage{amsmath}
\usepackage{amsthm}
\usepackage{amssymb}
\usepackage{graphicx}
\usepackage{sidecap}
\usepackage{caption}
\usepackage{subfig}
\usepackage{wrapfig}
\usepackage{psfrag}
\usepackage{mathrsfs}
\usepackage{tikz}
\usepackage{multicol}
\usepackage{pgfplots}
\usepackage{hyperref,enumitem}

\newtheorem{theorem}{Theorem}

\newtheorem{proposition}[theorem]{Proposition}
\newtheorem{corollary}[theorem]{Corollary}
\theoremstyle{definition}

\theoremstyle{remark}

\newtheorem*{oss*}{Remark}

\newcommand{\N}{\mathbb N}
\newcommand{\R}{\mathbb R}
\newcommand{\C}{\mathbb C}

\newcommand{\cA}{{\mathcal A}}

\newcommand{\cF}{{\mathcal F}}

\newcommand{\cJ}{{\mathcal J}}

\newcommand{\cQ}{{\mathcal Q}}
\newcommand{\cR}{{\mathcal R}}
\newcommand{\cS}{{\mathcal S}}

\newcommand{\cW}{{\mathcal W}}

\newcommand{\email}[1]{\protect\href{mailto:#1}{#1}}

\newcommand{\D}{{\rm d}}
\newcommand{\E}{{\rm e}}

\begin{document}
\selectlanguage{english}
\title{Cone-Adapted Shearlets and Radon Transforms}

\author{Francesca Bartolucci\thanks{Department of Mathematics, University of Genoa, Via Dodecaneso 35, 16146 Genova, Italy (\email{bartolucci@dima.unige.it}, \email{demari@dima.unige.it}, \email{devito@dima.unige.it}).}
\and Filippo De Mari\footnotemark[1]
\and Ernesto De Vito\footnotemark[1]}

\maketitle

\abstract{We show that the cone-adapted shearlet coefficients can be computed by means of the limited angle horizontal and vertical (affine) Radon transforms and the one-dimensional wavelet transform. This yields formulas that open new perspectives for the inversion of the Radon transform.
\vspace{2mm}


\section{Introduction}
\label{JSsec:introduction}
The inversion of the Radon transform is a classical ill-posed inverse problem and consists in reconstructing an unknown signal $f$ on $\R^2$ from its line integrals \cite{helgason99}. The Radon transform of a signal $f$ is a function on the affine projective space
$\mathbb P^1\times \R=\{\Gamma\mid \Gamma \text{ line of }\R^2\}$ whose value at a line is the integral of $f$ along that line. We label lines in the plane by pairs $(v,t)\in\R^2$ as $x+vy=t$ and we define the horizontal (affine) Radon transform $\mathcal{R} f:\R^2\to \C$ of any $f\in L^1(\R^2)$ by
\begin{equation*}
\mathcal{R} f (v,t)=  \int_{\R}f(t-v y,y) \,\D y, \qquad\text{a.e.}\ (v,t)\in\R^2.
\end{equation*}
This version of the Radon transform is proved to be particularly well-adapted to the structure of the classical shearlet transform, see \cite{bardemadeviodo} and \cite{gr11}. We recall that the key idea in shearlet analysis is to construct a family of analyzing functions 
\[
\{S_{b,s,a}\psi(x)=|a|^{-3/4}\psi(A_a^{-1}N_s^{-1}(x-b)):b\in\R^2,\, s\in\R,\, a\in\R^{\times}\}
\]
by translating, shearing and dilating a fixed initial function $\psi\in L^2(\R^2)$, called mother shearlet. Once we have this family of analyzing functions, we define the shearlet transform of any $f\in L^2(\R^2)$ by $\cS_{\psi}f(b,s,a)=\langle f,S_{b,s,a}\psi\rangle$. If $\psi$ satisfies the admissible condition \eqref{admissibleconditionshearlet} we can recover any signal $f\in L^2(\R^2)$ from its shearlet transform through the reconstruction formula 
\begin{equation}\label{intro}
 f = \int_{\R^{\times}}\int_\R \int_{\R^2}
 \cS_{\psi}f(b,s,a) \,     S_{b,s,a}\psi\ \D b\D s\frac{\D a}{|a|^3},
\end{equation}
where the integral converges in the weak sense.
In \cite{bardemadeviodo} we have shown that the classical shearlet transform can be realised by applying first the horizontal (affine) Radon transform, then by computing a one-dimensional wavelet transform and, finally, performing a one-dimensional convolution. 
This relation opens the possibility to recover a signal from its Radon transform by using the shearlet inversion formula \eqref{intro}, where the coefficients $\cS_{\psi}f(b,s,a)$ depend on $f$ only through its Radon transform. Thus, formula \eqref{intro} allows to reconstruct an unknown signal $f$ from it Radon transform $\cR f$ by computing the family of 
coefficients $\{\cS_{\psi}f(b,s,a)\}_{b\in\R^2,s\in\R,a\in\R^{\times}}$. Equation \eqref{intro} has a disadvantage if one wants to use it in applications since the shearing parameter $s$ is allowed to vary over a non-compact set. This gives rise to problems 
in the reconstruction of signals mostly concentrated on the $x$-axis since the energy of such signals is mostly concentrated in the coefficients $\cS_{\psi}f(b,s,a)$ as $s\to\infty$. The standard way to address this problem is so-called "shearlets on the cone" construction introduced by Kutyniok and Labate \cite{kula09} for classical admissible shearlets $\psi$ and then generalized by Grohs \cite{gr11} requiring weaker conditions on $\psi$. The basic idea in this construction is to decompose the signals as $f=P_{C}f+P_{C^{\bf v}}f$ previous to the analysis, where $P_{C}$ is the frequency projection on the horizontal cone $C=\left\{(\xi_1,\xi_2)\in\R^2:\left|\xi_2/\xi_1\right|\leq1\right\}$ and $P_{C^{\bf v}}$ is the projection on the vertical cone $C^{\bf v}=\left\{(\xi_1,\xi_2)\in\R^2:\left|\xi_1/\xi_2\right|\leq1\right\}$. Then, chosen a suitable window function $g$, the following reconstruction formula holds true:
\begin{align}\label{intro2}
\nonumber\|f\|^2&=\int_{\R^2}|\langle f,T_{b}g\rangle|^2\ {\rm d}b+\int_{-1}^{1}\int_{-2}^{2}\int_{\R^2}|\mathcal{S}_{\psi}[P_{C}f](b,s,a)|^2\ {\rm d}b{\rm d}s\frac{{\rm d}a}{|a|^3}\\
&+\int_{-1}^{1}\int_{-2}^{2}\int_{\R^2}|\mathcal{S}^{\bf v}_{\psi^{\bf v}}[P_{C^{\bf v}}f](b,s,a)|^2\ {\rm d}b{\rm d}s\frac{{\rm d}a}{|a|^3},
\end{align}
where $\cF\psi^{\bf v}(\xi_1,\xi_2)=\cF\psi(\xi_2,\xi_1)$ and the so-called vertical shearlet transform $\mathcal{S}^{\bf v}_{\psi^{\bf v}}f(b,s,a)$ is obtained  from the classical shearlet transform by switching the roles of the $x$-axis and the $y$-axis.
In formula \eqref{intro2}, $P_{C}f$ is reconstructed via the classical shearlet transform and $P_{C^{\bf v}}f$ via the vertical shearlet transform and this allows to restrict the shearing parameter $s$ over a compact interval.
In this paper, applying the "shearlets on the cone" construction to our results presented in \cite{bardemadeviodo}, we obtain for any $f\in L^1(\R^2)\cap L^2(\R^2)$ a reconstruction formula of the form \eqref{intro2}, i.e. where both the scale parameter $a$ and the shearing parameter $s$ range over compact intervals, and where the coefficients depend on $f$ only through its Radon transform. Precisely, we show that the shearlet coefficients $\mathcal{S}_{\psi}[P_{C}f](b,s,a)$ depend on $f$ through its (affine) horizontal Radon transform $\cR f(v,t)$ and the action of the projection $P_C$ on $f$ turns into the restriction of the directional parameter $v$ over the compact interval [-1,1]. Analogously, the vertical shearlet coefficients $\mathcal{S}^{\bf v}_{\psi^{\bf v}}[P_{C^{\bf v}}f](b,s,a)$ depend on the limited angle (affine) vertical Radon transform $\cR^{\bf v}f(v,t)$, $|v|\leq1$, obtained by switching the roles of the $x$-axis and the $y$-axis in the affine parametrization. Therefore, equation \eqref{intro2} allows to reconstruct an unknown signal $f\in L^1(\R^2)\cap L^2(\R^2)$ from its Radon transform by computing the family of coefficients $\{\langle f,T_{b}g\rangle,\mathcal{S}_{\psi}[P_{C}f](b,s,a),\mathcal{S}^v_{\psi^{v}}[P_{C^{\bf v}}f](b,s,a)\}_{b\in\R^2,s\in\R,a\in\R^{\times}}$ by means of Theorem~\ref{reconstructionteo}.
The different contributions $\cR f(v,t)$ and $\cR^{\bf v}f(v,t)$, $|v|\leq1$, reconstruct the frequency projections $P_Cf$ and $P_{C^{\bf v}}f$, respectively. Finally, in Section~\ref{sec:final} we generalize reconstruction formula \eqref{intro2} by applying to $f$ 
localization operators different from $P_C$ and $P_{C^{\bf v}}$ in order to avoid artificial singularities in the reconstructed signal. The paper is organised as it follows. In Section~\ref{sec:overview} we recall the notion of wavelet transform, shearlet transform and Radon transform and part of the results in \cite{bardemadeviodo}. In Section~\ref{sec:main} we present the main results. Finally, in Section~\ref{sec:final} we generalize the results presented in Section \ref{sec:main}.
\section{Preliminaries}\label{sec:overview}
In this section we introduce the notation and we recall the definition and the main properties of the three main ingredients, namely the wavelet transform, the shearlet transform and the Radon transform. Then, we recall part of the results in \cite{bardemadeviodo} which show how these three classical transforms are related.
\subsection{Notation}
We briefly introduce the notation. We set $\R^{\times}=\R\setminus\{0\}$. The Euclidean norm of a vector
$v\in\R^d$ is denoted by $|v|$ and its scalar product with $w\in\R^d$ by $v\cdot w$. For any $p\in[1,+\infty]$ we denote
by $L^p(\R^d)$ the Banach space of functions $f\colon\R^d\rightarrow\C$ that are $p$-integrable with respect to the Lebesgue measure $\D x$
and, if $p=2$, the corresponding scalar product and norm are
$\langle\cdot,\cdot\rangle$ and $\|\cdot\|$, respectively.
The Fourier transform is denoted by $\mathcal F$ both on
$L^2(\R^d)$ and on  $L^1(\R^d)$, where it is 
defined by
\[
\mathcal F f({\xi}\,)= \int_{\R^d} f(x) \E^{-2\pi i\,
  {\xi}\cdot x } \D{x},\qquad f\in L^1(\R^d).
\] 
If $G$ is a locally compact group, we denote by $L^2(G)$ the Hilbert
space of square-integrable functions with respect to a left Haar
measure on $G$. If $A\in M_{d}(\R)$, the vector space of square $d\times d$ matrices with real entries, $^t\! A$ denotes its transpose and we denote the (real) general linear group of size $d\times d$ by ${\rm GL}(d,\R)$. 
Finally, the translation operator acts on a
function $f:\R^d\to \C$ as
$
T_bf(x)=f(x-b),
$
for any $b\in\R^d$. 
\subsection{The wavelet transform}
The one-dimensional affine group $\mathbb{W}$ is the semidirect product $\R\rtimes\R^{\times}$ with group operation
\[
(b,a)(b',a')=(b+ab',aa')
\]
and left Haar measure $|a|^{-2}\D b\D a$. It acts on $L^2(\R)$ by means of the square-integrable representation 
\[
W_{b,a}f(x)=|a|^{-\frac{1}{2}}f\left(\frac{x-b}{a}\right).
\]
The wavelet transform is then $\cW_{\psi}f(b,a)=\langle f, W_{b,a}\psi\rangle$, which is a multiple of an isometry provided that $\psi\in L^2(\R)$ satisfies the admissibility condition, namely the Calder\'on equation,
\textcolor{black}{\begin{equation}\label{calderon}
0<\int_{\R}\frac{|\cF\psi(\xi)|^2}{|\xi|}\D\xi<+\infty
\end{equation}}
and, in such a case, $\psi$ is called a one-dimensional wavelet. 
\subsection{The shearlet transform}
In this subsection we start presenting the standard shearlet group introduced and studied in
 \cite{lawakuwe05,dastte10} and further investigated in \cite{codemanowtab06,codemanowtab10} as an extension of the Heisenberg group with homogeneous dilations
 and in \cite{cortab13} as a subgroup of the symplectic group. Furthermore, the standard shearlet group has been extended by F$\"u$hr in \cite{fu98,futo16} where the generalized shearlet dilation groups
are introduced.
The (parabolic) shearlet group $\mathbb{S}$ is the semidirect product
of $\R^2$ with the closed subgroup $K=\{N_sA_a\in {\rm GL}(2,\R):s\in\R,a\in\R^{\times}\}$ where 
\[
N_s=\left[\begin{matrix}1 & -s\\ 0 & 1\end{matrix}\right],\qquad A_a=a\left[\begin{matrix}1 & 0\\ 0 & |a|^{-1/2}\end{matrix}\right].
\]
We can identify the element $N_sA_a$ with the pair $(s,a)$ and write $(b,s,a)$ for the elements in $\mathbb{S}$. With this identification the product law amounts to 
\[
(b,s,a)(b',s',a')=(b+N_sA_ab',s+|a|^{1/2}s',aa').
\]
A left Haar measure of $\mathbb{S}$ is 
\[
{\rm d}\mu(b,s,a)=|a|^{-3}{\rm d}b{\rm d}s{\rm d}a, 
\]
with ${\rm d}b$, ${\rm d}s$ and ${\rm d}a$ the Lebesgue measures on $\R^2$, $\R$ and $\R^{\times}$, respectively. 
The group $\mathbb{S}$ acts on $L^2(\R^2)$ via the square-integrable representation 
\begin{equation*}
S_{b,s,a}f(x)=|a|^{-3/4}f(A_a^{-1}N_s^{-1}(x-b))
\end{equation*}
and the shearlet transform $\cS_{\psi}f(b,s,a)=\langle f,S_{b,s,a}\psi\rangle$ is a multiple of an isometry from $L^2(\R^2)$ into $L^2(\mathbb{S},{\rm d}\mu)$ provided that $\psi\in L^2(\R^2)$ satisfies the admissible condition
\begin{equation}\label{admissibleconditionshearlet}
0<C_{\psi}=\int_{\R^2}\frac{|\cF\psi(\xi)|^2}{|\xi_1|^2}\D\xi<+\infty,
\end{equation}
where $\xi=(\xi_1,\xi_2)\in\R^2$  \cite{dahlke2008}, or equivalently 
\begin{equation*}
\int_{\R}\int_{\R}|\cF\psi(A_a{^{t}\!N_s}\xi)|^2\D s\frac{\D a}{|a|^{3/2}}=C_{\psi},\qquad \text{for a.e. $\xi\in\R^2/\{0\}$}.
\end{equation*}
Furthermore,  in such a case, we have the reconstruction formula 
\begin{equation}\label{reconstructionformulashearlet}
 f = \frac{1}{C_{\psi}}\int_{\R^{\times}}\int_\R \int_{\R^2}
 \cS_{\psi}f(b,s,a) \,     S_{b,s,a}\psi\ \D b\D s\frac{\D a}{|a|^3},
\end{equation}
where the integral converges in the weak sense, and
\begin{equation}\label{reconstructionformulashearlet2}
 \|f\|^2 = \frac{1}{C_{\psi}}\int_{\R^{\times}}\int_\R \int_{\R^2}
 |\cS_{\psi}f(b,s,a)|^2 \,  \D b\D s\frac{\D a}{|a|^3}.
\end{equation}
From now on, when we consider an admissible vector $\psi$, we suppose $C_{\psi}=1$.

Although the shearlet transform exhibits an elegant group structure and is based on the theory of square integrable representations, the reconstruction formula \eqref{reconstructionformulashearlet} 
has one disadvantage: the shearing parameter ranges over a non-compact set and this can constitute a limitation in applications. For example, if $f$ is the delta distribution supported on the $x$-axis, a classical model for an edge in an image, the high amplitude shearlet coefficients, i.e. the shearlet coefficients in which the energy of the signal is mostly concentrated, 
correspond to the shearlet coefficients $\cS_{\psi}f(b,s,a)$ as $s\to\infty$ \cite{kula09}. In order to avoid this problem Kutyniok and Labate \cite{kula09} proposed the "shearlets on the cone" construction which leads to a reconstruction formula of the form \eqref{reconstructionformulashearlet} in which both the scale parameter $a$ and the shearing parameter $s$ are restricted over compact sets. We briefly recall this construction.\\

Let $f\in L^2(\R^2)$. We consider the horizontal and vertical cones in the frequency plane
\begin{equation}\label{cones}
C=\left\{(\xi_1,\xi_2)\in\R^2:\left|\frac{\xi_2}{\xi_1}\right|\leq1\right\},\qquad C^{\bf v}=\left\{(\xi_1,\xi_2)\in\R^2:\left|\frac{\xi_1}{\xi_2}\right|\leq1\right\}.
\end{equation}
\begin{figure}[h]
\centering
	\includegraphics[scale=.4]{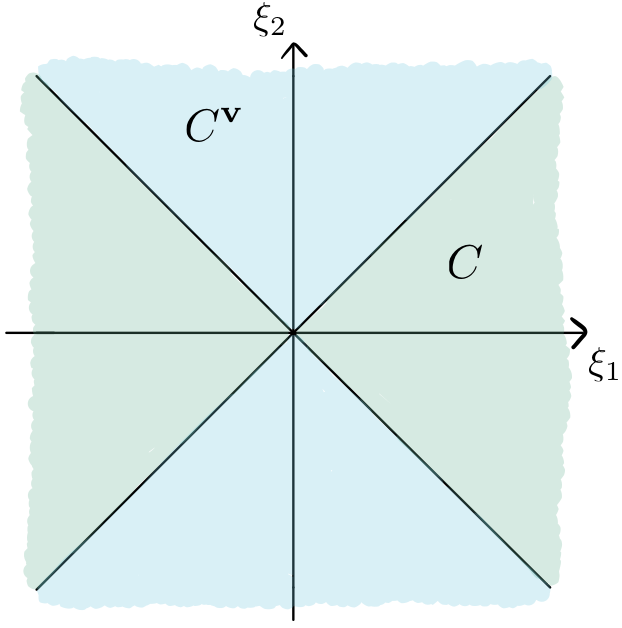}
	%
	%
	\caption{In the "shearlets on the cone" construction the frequency plane is divided in two cones $C$ and $C^{\bf v}$. In formula \eqref{reconstructionghros}, $P_{C}f$ is
	reconstructed via the classical shearlet transform and $P_{C^{\bf v}}f$ via the so-called vertical shearlet transform.}
	\label{JSfig:1}       
\end{figure} 
If $D$ is a region in the plane, we denote by $\chi_{D}$ its characteristic function, i.e.
\[
\chi_D(\xi)=\begin{cases}1 & \text{if}\,\,\,\xi\in D\\ 0 &\text{if}\,\,\,\xi\not\in D\end{cases}
\]
and we define the frequency projections of $f$ onto $C$ and $C^{\bf v}$ by 
\begin{align}\label{horizontalprojection}
\nonumber\cF(P_{C}f)(\xi_1,\xi_2)=\cF f(\xi_1,\xi_2)\chi_{C}(\xi_1,\xi_2)\\
\\
\nonumber\cF(P_{C^{\bf v}}f)(\xi_1,\xi_2)=\cF f(\xi_1,\xi_2)\chi_{C^{\bf v}}(\xi_1,\xi_2)
\end{align}
respectively. 

We need a modified version of the continuous shearlet transform obtained by switching the roles of the $x$-axis and the $y$-axis. 
We introduce the vertical shearlet representation
\begin{equation*}
S_{b,s,a}^{\bf v}f(x)=|a|^{-3/4}f(\tilde{A_a}^{-1}\tilde{N_s}^{-1}(x-b))
\end{equation*}
where 
\[
\tilde{N_s}=\left[\begin{matrix}1 & 0\\ -s & 1\end{matrix}\right],\qquad \tilde{A_a}=a\left[\begin{matrix}|a|^{-1/2} & 0\\ 0 & 1\end{matrix}\right],
\]
and the associated vertical shearlet transform $\cS^{\bf v}_{\psi}f(b,s,a)=\langle f,S_{b,s,a}^{\bf v}\psi\rangle$.\\

\textcolor{black}{Reconstruction formulas of the form \eqref{reconstructionghros} were firstly proved by Labate and Kutinyok  \cite{kula09} 
for classical admissible shearlets $\psi$ and then generalized by Grohs \cite{gr11} requiring weaker conditions on $\psi$. We have chosen to present our results within the second approach.}
We fix $\psi\in L^2(\R^2)$ satisfying the admissibility condition \eqref{admissibleconditionshearlet}. We require that $\psi$ is a smooth function with infinitely directional vanishing moments in the $x_1$-direction \cite{gr11}, that is 
\[
\int_{\R}x_1^N\psi(x_1,x_2)\D x_1=0,\qquad \text{for all}\, x_2\in\R,\, N\in\N.
\]
Finally, we define 
\begin{equation*}
\cF\psi^{\bf v}(\xi_1,\xi_2)=\cF\psi(\xi_2,\xi_1).
\end{equation*}
Then, we have the following result. 
\begin{theorem}\label{reconstructionteo1}
For any $f\in L^2(\R^2)$, we have the reconstruction formula 
\begin{align}\label{reconstructionghros}
\nonumber\|f\|^2&=\int_{\R^2}|\langle f,T_{b}g\rangle|^2\ {\rm d}b+\int_{-1}^{1}\int_{-2}^{2}\int_{\R^2}|\mathcal{S}_{\psi}[P_{C}f](b,s,a)|^2\ {\rm d}b{\rm d}s\frac{{\rm d}a}{|a|^3}\\
&+\int_{-1}^{1}\int_{-2}^{2}\int_{\R^2}|\mathcal{S}^{\bf v}_{\psi^{\bf v}}[P_{C^{\bf v}}f](b,s,a)|^2\ {\rm d}b{\rm d}s\frac{{\rm d}a}{|a|^3},
\end{align}
with $g\in C^\infty(\R^2)$ such that for all $\xi\in\R^2$
\begin{align}\label{W}
\nonumber|\cF g(\xi)|^2&+\chi_{C}(\xi)\int_{-1}^{1}\int_{-2}^{2}|\cF\psi(A_a{^t\!N_s}\xi)|^2\ {\rm d}s\frac{{\rm d}a}{a^{3/2}}\\
&+\chi_{C^{\bf v}}(\xi)\int_{-1}^{1}\int_{-2}^{2}|\cF\psi^{\bf v}(\tilde{A_a}N_s\xi)|^2\ {\rm d}s\frac{{\rm d}a}{a^{3/2}}=1.
\end{align}
\end{theorem}
\textcolor{black}{
We refer to \cite{gr11} and \cite[Chapter 2]{kula12} for the proof. 
} 

\subsection{The Radon transform}
\textcolor{black}{The Radon transform of a signal $f$ is a function on the affine projective space
$\mathbb P^1\times \R=\{\Gamma\mid \Gamma \text{ line of }\R^2\}$ whose value at a line is the integral of $f$ along that line. It is usually defined by parametrizing the lines by pairs $(\theta,t)\in[0,\pi)\times\R$ as 
\[ \Gamma_{\theta,t}= \{ (x,y)\in \R^2\mid \cos{\theta}x+ \sin{\theta}y=t\}, \]
see \cite{helgason99}.} We label the normal vector to a line by
affine coordinates, that is 
\[ \Gamma_{v,t}= \{ (x,y)\in \R^2\mid x+ v y=t\}, \]
see Figure~\ref{JSfig:2}. With this parametrisation,
the horizontal lines 
can not be represented, but they constitute a negligible 
set with respect to the natural measure on $\mathbb P^1\times \R$.  
The \textcolor{black}{horizontal (affine)} Radon transform of any $f\in L^1(\R^2)$ is the function $\mathcal{R} f:\R^2\to \C$ defined by
\begin{equation*}
\mathcal{R} f (v,t)=  \int_{\R}f(t-v y,y) \,\D y, \qquad\text{a.e.}\ (v,t)\in\R^2.
\end{equation*}
\textcolor{black}{The choice of the affine parametrization} is particularly 
well-adapted to the mathematical structure of the shearlet transform, see also \cite{gr11}. 
It is possible to extend 
$\mathcal{R}$ to $L^2(\R^2)$ as a unitary map. However, this raises some
technical issues. First, consider the dense subspace of $L^2(\R^2)$
\[\mathcal D = \{ g\in L^2(\R^2)\mid \hspace{-.1cm}
    \int_{\R^2}
  |\tau| |(I\otimes \mathcal
    F)g(v,\tau)|^2\D v\D\tau<+\infty   \}, \]
where $I:L^2(\R)\to L^2(\R)$ is the identity operator, and then define the self-adjoint unbounded
operator $\mathcal{J}: \mathcal D\to 
  L^2(\R^2)$ by
\begin{equation*}(I\otimes \mathcal F)\mathcal{J} F (v,\tau)= |\tau|^{\frac{1}{2}}
(I\otimes \mathcal F) F (v,\tau),\qquad \text{a.e. }(v,\tau)\in \R^2,
\end{equation*}
which is a Fourier multiplier with respect to the  second variable. Then,
it is not hard to show that for all $f$ in the dense subspace of $ L^2(\R^2)$
\[
\cA=\{ f\in L^1(\R^2)\cap L^2(\R^2)\mid\! 
\int_{\R^2}\frac{|\mathcal F f(\xi_1,\xi_2)|^2}{|\xi_1|}\D\xi_1\D\xi_2<+\infty\},
\]
 the Radon
transform $\mathcal{R} f$ belongs to $\mathcal D$ and the map
\[
f \longmapsto \mathcal{J} \mathcal{R} f 
\]
from $\cA$ to $L^2(\R^2)$ extends to a unitary
map, denoted by $\mathcal{Q}$, from $L^2(\R^2)$ onto itself. We refer to \cite{helgason99} and \cite{bardemadeviodo} for technical details. We need the following version of the Fourier slice theorem.
\textcolor{black}{\begin{corollary}\label{cfst}
Let $f\in L^1(\R^2)\cap L^2(\R^2)$. For almost every $v\in\R$ the function $\cR f(v,\cdot)$ is in $L^2(\R)$ and satisfies
\begin{equation}\label{fst}
\cF(\cR f (v,\cdot))(\tau)= \mathcal F f(\tau, \tau v).
\end{equation}
Furthermore, for any $f\in L^2(\R^2)$
\begin{equation}\label{fstg}
\cF(\cQ f (v,\cdot))(\tau)= |\tau|^{\frac{1}{2}} \mathcal F f(\tau, \tau v), \qquad\text{a.e.}\ (v,t)\in\R^2.
\end{equation}
\end{corollary}
In \eqref{fst} and \eqref{fstg} the Fourier transform on the right hand side is in $\R^2$, whereas the operator $\cF$ on the left hand side is one-dimensional and acts on the variable $t$. We repeat this slight abuse of notation in other formulas below. The first statement in Corollary~\ref{cfst} is the classical Fourier slice theorem \cite{helgason99} adapted to the horizontal (affine) Radon transform \cite{bardemadeviodo}. The proof of the second part in Corollary~\ref{cfst} is not trivial because $\cQ$ cannot be written as the composition $\cJ\cR$ for arbitrary $f\in L^2(\R^2)$ and is based on \eqref{fst} and the fact that $\cJ$ is a Fourier
multiplier (see Appendix B in \cite{bardemadeviodo}).\\
}
\textcolor{black}{
We repeat the construction above by exchanging 
the role of the $x$-axis and the $y$-axis} and we parametrize the lines in the plane, except the vertical ones, by pairs $(v,t)\in\R\times\R$ as follows
\[ \Gamma_{v,t}= \{ (x,y)\in \R^2\mid vx+ y=t\}, \]
see Figure~\ref{JSfig:2}. The \textcolor{black}{vertical (affine)} Radon transform of any $f\in L^1(\R^2)$ is the function $\mathcal{R}^{\bf v} f:\R^2\to \C$ defined by
\begin{equation*}
\mathcal{R}^{\bf v} f (v,t)=  \int_{\R}f(x,t-vx) \,\D x, \qquad\text{a.e.}\ (v,t)\in\R^2.
\end{equation*}
\begin{figure}[h]
\centering
	\includegraphics[scale=.4]{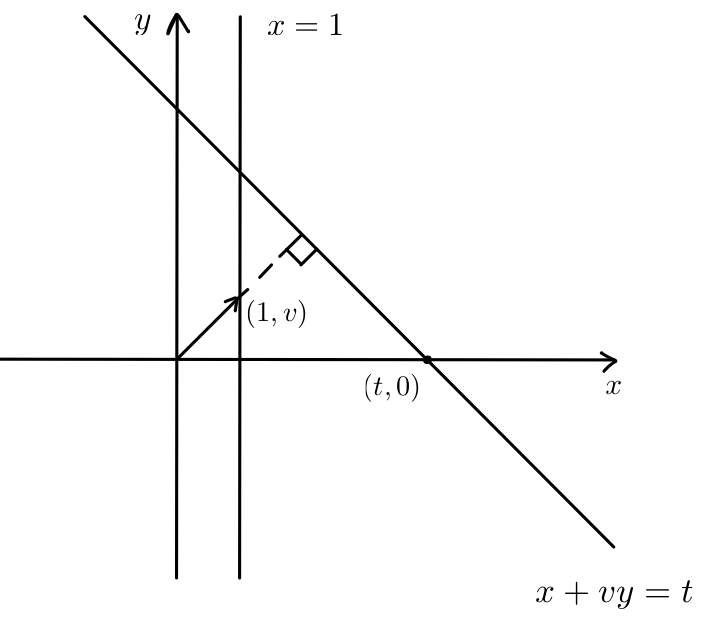}
	%
	%
	\caption{The \textcolor{black}{horizontal} Radon transform is defined by labeling the normal vector to a line, except the horizontal ones, by affine coordinates (figure on top): $v$ parametrizes
	the slope of a line and $t$ its intersection with the $x$-axis. The \textcolor{black}{vertical} Radon transform is obtained just switching the roles of the $x$-axis and the $y$-axis in the previous parametrization.}
	\label{JSfig:2}       
\end{figure} 
\textcolor{black}{As for the 
horizontal Radon transform, define the dense subspace ${\cA}^{\bf v}$ of $ L^2(\R^2)$ by 
\[
{\cA}^{\bf v}=\{ f\in L^1(\R^2)\cap L^2(\R^2)\mid\! 
\int_{\R^2}\frac{|\mathcal F f(\xi_1,\xi_2)|^2}{|\xi_2|}\D\xi_1\D\xi_2<+\infty\}.
\] 
Then, the composite operator $\mathcal{J} \mathcal{R}^{\bf v} :{\cA}^{\bf v}\to L^2(\R^2)$ extends to a unitary map ${\cQ}^{\bf v}$ from $L^2(\R^2)$ onto itself.}
\subsection{The Radon transform intertwines wavelets and shearlets}
We recall part of the results in \cite{bardemadeviodo}. We fix $\psi\in L^2(\R^2)$ of the form 
\begin{equation}\label{eq:2}
\cF\psi(\xi_1,\xi_2)=\cF\psi_1(\xi_1)\cF\psi_2\left(\frac{\xi_2}{\xi_1}\right),
\end{equation}
with $\psi_1\in L^2(\R)$ satisfying the conditions 
\begin{equation}\label{eq:4}
0<\int_{\R}\frac{|\cF\psi_1(\tau)|^2}{|\tau|}\ {\rm d}\tau<+\infty,\qquad\int_{\R}|\tau||\cF\psi_1(\tau)|^2\ {\rm d}\tau<+\infty
\end{equation}
and $\psi_2\in L^2(\R)$.
Then, $\psi$ satisfies the admissible condition \eqref{admissibleconditionshearlet} and the function $\phi_1\in L^2(\R)$ defined by 
\begin{equation}\label{phi}
\mathcal{F}\phi_1(\tau)=|\tau|^\frac{1}{2}\mathcal{F}\psi_{1}(\tau)
\end{equation}
is a one-dimensional wavelet, i.e. it satisfies condition \eqref{calderon}. \textcolor{black}{
\begin{theorem}\label{teocentrale}
For any $f\in L^2(\R^2)$ and $(x,y,s,a)\in \R^2\times\R\times\R^{\times}$,
\begin{equation}\label{first}
\mathcal{S}_{\psi}f(x,y,s,a)=|a|^{-\frac{1}{4}}\int_{\mathbb{R}}\mathcal{W}_{\phi_1}(\cQ f(v,\cdot))(x+vy,a)\overline{\phi_2\left(\frac{v-s}{|a|^{1/2}}\right)}\ {\rm d}v
\end{equation}
and, analogously for the vertical shearlet transform,
\begin{equation*}
\mathcal{S}^{\bf v}_{\psi^{\bf v}}f(x,y,s,a)=|a|^{-\frac{1}{4}}\int_{\mathbb{R}}\mathcal{W}_{\phi_1}({\cQ}^{\bf v}f(v,\cdot))(vx+y,a)\overline{\phi_2\left(\frac{v-s}{|a|^{1/2}}\right)}\ {\rm d}v,
\end{equation*}
where $\phi_1$ is the one-dimensional wavelet defined by \eqref{phi} and $\phi_2=\cF\psi_2$.
\end{theorem}
We refer to \cite{bardemadeviodo} for the proof. }

\section{Cone-adapted shearlets and Radon transforms }\label{sec:main}
\textcolor{black}{Equation~\eqref{reconstructionformulashearlet2} together with formula \eqref{first} allows to reconstruct an unknown signal $f$ from its unitary Radon transform $\cQ f$ but it is difficult to implement in applications since $\cQ$ involves both a limit and the pseudo-differential operator $\mathcal J$.
Furthermore, in the reconstruction formula~\eqref{reconstructionformulashearlet2} the shearing 
parameter $s$ is allowed to range over $\R$ and this can give rise to the problems discussed above.} The aim of this paper is to obtain a reconstruction formula of the form \eqref{reconstructionghros}, i.e. where both the scale and the shearing parameters belong to compact intervals, where the shearlet coefficients depend on $f$ only through its Radon transform and do not involve the operator $\cJ$ applied to the signal.\\

We fix an admissible vector $\psi$ of the form \eqref{eq:2} satisfying conditions \eqref{eq:4} \textcolor{black}{and such that $\psi_1$ satisfies the further condition
\begin{equation}
\label{newcondition}
\int_{\mathbb{R}}|\tau|^{2}|\mathcal{F}\psi_1(\tau)|^2\ {\rm d}\tau<+\infty.
\end{equation}
\begin{proposition}\label{propositionfirst}
For any $f\in L^1(\R^2)\cap L^2(\R^2)$ and $(x,y,s,a)\in \R^2\times\R\times\R^{\times}$,
\[
\mathcal{S}_{\psi}[P_{C}f](x,y,s,a)=|a|^{-\frac{3}{4}}\int_{-1}^1\mathcal{W}_{\chi_1}(\cR f(v,\cdot))(x+vy,a)\overline{\phi_2\left(\frac{v-s}{|a|^{1/2}}\right)}\ {\rm d}v,
\]
\[
\mathcal{S}^{\bf v}_{\psi^{\bf v}}[P_{C^{\bf v}}f](x,y,s,a)=|a|^{-\frac{3}{4}}\int_{-1}^1\mathcal{W}_{\chi_1}(\cR^{\bf v}f(v,\cdot))(vx+y,a)\overline{\phi_2\left(\frac{v-s}{|a|^{1/2}}\right)}\ {\rm d}v,
\]
where $\mathcal{F}\chi_1(\tau)=|\tau|\mathcal{F}\psi_{1}(\tau)$ and $\phi_2=\cF\psi_2$.
\end{proposition}
}
\begin{proof}
We take a function $f\in L^1(\R^2)\cap L^2(\R^2)$ and we consider its frequency projection $P_{C}f$ on the horizontal cone $C$ defined by \eqref{cones} and \eqref{horizontalprojection}. Since $P_{C}f$ belongs to $L^2(\R^2)$, we can apply formula \eqref{first} and we obtain 
\begin{equation}\label{eq:1}
\mathcal{S}_{\psi}[P_{C}f](x,y,s,a)=|a|^{-\frac{1}{4}}\int_{\mathbb{R}}\mathcal{W}_{\phi_1}(\cQ [P_{C}f](v,\cdot))(x+vy,a)\overline{\phi_2\left(\frac{v-s}{|a|^{1/2}}\right)}\ {\rm d}v,
\end{equation}
where $\phi_1$ is the admissible wavelet defined by \eqref{phi} and $\phi_2=\cF\psi_2$. We consider the functions $t\mapsto\cQ [P_{C}f](v,t)$ in equation \eqref{eq:1}. By Corollary~\ref{cfst} and the definition of $P_{C}f$, we have
\begin{equation}\label{Qprojection}
\cF(\cQ [P_{C}f](v,\cdot))(\tau)=|\tau|^{\frac{1}{2}}\cF(P_{C}f)(\tau,\tau v)=|\tau|^{\frac{1}{2}}\cF f(\tau,\tau v)\chi_{C}(\tau,\tau v),
\end{equation}
for almost every $(v,\tau)\in\R^2$. Furthermore, by the definition of the horizontal cone $C$, the function $\tau\mapsto\chi_{C}(\tau,\tau v)$ is identically one if $|v|\leq1$ and zero otherwise. Thus, \eqref{Qprojection} becomes
\begin{equation}\label{Qprojection2}
\cF(\cQ [P_{C}f](v,\cdot))(\tau)=
\begin{cases}
|\tau|^{\frac{1}{2}}\cF f(\tau,\tau v) & \text{if}\,\,\, |v|\leq1\\
0 & \text{if}\,\,\, |v|>1
\end{cases}.
\end{equation}
From now on we consider the case $|v|\leq1$. Since $f\in L^1(\R^2)\cap L^2(\R^2)$, Corollary~\ref{cfst} and equation \eqref{Qprojection2} imply that for almost all $v\in\R$, $\cR f(v,\cdot)\in L^2(\R)$ and  
\begin{equation}\label{Qprojection3}
\cF(\cQ [P_{C}f](v,\cdot))(\tau)=|\tau|^{\frac{1}{2}}\cF f(\tau,\tau v)=|\tau|^{\frac{1}{2}}\cF\cR f(v,\cdot)(\tau).
\end{equation}
Since $\tau\mapsto\cF(\cQ [P_{C}f](v,\cdot))(\tau)\in L^2(\R)$ for almost all $v\in\R$, equality \eqref{Qprojection3}
implies that $\cR f(v,\cdot)$ is in the domain of the differential operator $\cJ_0:L^2(\R)\to L^2(\R)$ defined as
\begin{equation}\label{rieszoperator}
\mathcal{F}\mathcal{J}_0g(\tau)=|\tau|^{\frac{1}{2}}\mathcal{F}g(\tau),
\end{equation}
and, by the definition of $\mathcal{J}_0$, 
\begin{align*}
\cQ [P_{C}f](v,\cdot)&=\cJ_0\cR f(v,\cdot).
\end{align*}
Since $\cJ_0$ is a self-adjoint operator, the wavelet coefficients in \eqref{eq:1} become
\begin{align*}
\mathcal{W}_{\phi_1}(\cQ [P_{C}f](v,\cdot))(x+vy,a)&=\langle\cQ [P_{C}f](v,\cdot),W_{x+vy,a}\phi_1\rangle_2\\
&=\langle\cJ_0\cR f(v,\cdot),W_{x+vy,a}\phi_1\rangle_2\\
&=\langle\cR f(v,\cdot),\cJ_0W_{x+vy,a}\phi_1\rangle_2\\
&=|a|^{-\frac{1}{2}}\langle\cR f(v,\cdot),W_{x+vy,a}\cJ_0\phi_1\rangle_2\\
&=|a|^{-\frac{1}{2}}\mathcal{W}_{\cJ_0\phi_1}(\cR f(v,\cdot))(x+vy,a),
\end{align*}
\textcolor{black}{by taking into account that
\[
\cJ_0W_{b,a}=|a|^{-\frac{1}{2}}W_{b,a}\cJ_0,
\] 
for any $(b,a)\in\R\times\R^\times$. We set $\chi_1=\cJ_0\phi_1=\cJ_0^2\psi_1$, that is  
\begin{equation*}
\mathcal{F}\chi_1(\tau)=|\tau|\mathcal{F}\psi_{1}(\tau),
\end{equation*}
which is well-defined since $\psi_1$ satisfies \eqref{newcondition}.}
From the above calculations, we can conclude that for almost every $v\in\R$
\[
\mathcal{W}_{\phi_1}(\cQ [P_{C}f](v,\cdot))(x+vy,a)=
\begin{cases}
|a|^{-\frac{1}{2}}\mathcal{W}_{\chi_1}(\cR f(v,\cdot))(x+vy,a) & |v|\leq1\\
0 & |v|>1
\end{cases}
\]
and formula \eqref{eq:1} becomes 
\begin{align}\label{coeffh}
\mathcal{S}_{\psi}[P_{C}f](x,y,s,a)=|a|^{-\frac{3}{4}}\int_{-1}^1\mathcal{W}_{\chi_1}(\cR f(v,\cdot))(x+vy,a)\overline{\phi_2\left(\frac{v-s}{|a|^{1/2}}\right)}\ {\rm d}v.
\end{align}
Using the same arguments as in the case of the horizontal cone, we obtain the following formula for the vertical shearlet transform
\begin{align}\label{coeffv}
\mathcal{S}^{\bf v}_{\psi^{\bf v}}[P_{C^{\bf v}}f](x,y,s,a)&=|a|^{-\frac{3}{4}}\int_{-1}^1\mathcal{W}_{\chi_1}(\cR^{\bf v}f(v,\cdot))(vx+y,a)\overline{\phi_2\left(\frac{v-s}{|a|^{1/2}}\right)}\ {\rm d}v.
\end{align}
This completes the proof.
\end{proof}

It is worth observing that  formulas \eqref{coeffh} and \eqref{coeffv} turn the action of the frequency projections $P_{C}$  and $P_{C^{\bf v}}$ on $f$ into the restriction of the interval over which we integrate the directional variable $v$
and so, \eqref{coeffh} and \eqref{coeffv} eliminate the need to perform a frequency projection on $f$ prior to the analysis. Furthermore, as a consequence, the shearlet coefficients $\mathcal{S}_{\psi}[P_{C}f](b,s,a)$ and $\mathcal{S}^{\bf v}_{\psi^{\bf v}}[P_{C^{\bf v}}f](b,s,a)$ depend on $f$ through its limited angle (affine) horizontal and vertical Radon transforms $\cR f(v,t)$ and $\cR^{\bf v}f(v,t)$, with $|v|\leq1$, respectively.

Finally, let us show that also the first integral in the right hand side of reconstruction formula \eqref{reconstructionghros} may be expressed in terms of $\cR f$ only. 
\begin{proposition}\label{hcoeff}
For any $f\in L^1(\R^2)\cap L^2(\R^2)$ and for any smooth function $g$ in $L^1(\R^2)\cap L^2(\R^2)$ we have that 
\begin{align*}
\langle f,T_{b}g\rangle=\int_{\R}\langle\mathcal{R}f(v,\cdot),T_{n(v)\cdot b}\zeta(v,\cdot)\rangle{\rm d}v,
\end{align*}
for any $b\in\R^2$, where $\zeta=\cJ_0^2\mathcal{R}g$ and $n(v)={^t\!(1,v)}$.

\end{proposition}
\begin{proof}
We take a function $f\in L^1(\R^2)\cap L^2(\R^2)$ and we consider a smooth function $g\in L^1(\R^2)\cap L^2(\R^2)$.  We readily derive
\begin{align*}
\langle f,T_{b}g\rangle=&\langle{\mathcal{Q}}f,{\mathcal{Q}}T_{b}g\rangle
=\int_{\R}\langle{\mathcal{Q}}f(v,\cdot),{\mathcal{Q}}T_{b}g(v,\cdot)\rangle{\rm d}v.
\end{align*}
Since $f$ and $g$ are in $L^1(\R^2)\cap L^2(\R^2)$, Corollary~\ref{cfst} implies that for almost all $v\in\R$, $\cR f(v,\cdot)$ and $\cR g(v,\cdot)$ are square-integrable functions and
\begin{align*}
\langle f,T_{b}g\rangle=&\int_{\R}\langle\mathcal{J}_0\mathcal{R}f(v,\cdot),\mathcal{J}_0\mathcal{R}T_{b}g(v,\cdot)\rangle{\rm d}v,
\end{align*}
where we recall that $\cJ_0$ is the differential operator defined by \eqref{rieszoperator}.
By the behavior of the horizontal Radon transform under translations \cite{helgason99} and since the operator $\cJ_0$ commutes with translations, denoting $n(v)={^t\!(1,v)}$, we have that 

\begin{align*}
\cJ_0\mathcal{R}T_{b}g(v,t)=\cJ_0(I\otimes T_{n(v)\cdot b})\mathcal{R}g(v,t)=(I\otimes T_{n(v)\cdot b})\cJ_0\mathcal{R}g(v,t).
\end{align*}
We need to choose $g$ in such a way that $\cJ_0\mathcal{R}g(v,\cdot)$ is in the domain of the operator $\cJ_0$ for almost every $v\in\R$. Assuming this, the same property holds true for $T_{n(v)\cdot b}\cJ_0\mathcal{R}g(v,\cdot)$ by the translation invariance of $\text{dom}\,\cJ_0$ and we obtain  
\begin{align}\label{Wcoeff1}
\nonumber\langle f,T_{b}g\rangle=&\int_{\R}\langle\cJ_0\mathcal{R}f(v,\cdot),T_{n(v)\cdot b}\cJ_0\mathcal{R}g(v,\cdot)\rangle{\rm d}v\\
=&\int_{\R}\langle\mathcal{R}f(v,\cdot),T_{n(v)\cdot b}\cJ_0^2\mathcal{R}g(v,\cdot)\rangle{\rm d}v.
\end{align}
It is worth observing that the extra assumption that $\cJ_0\mathcal{R}g(v,\cdot)$ is in the domain of $\cJ_0$ for almost every $v\in\R$ is always satisfied. Indeed, by the definition of $\cJ_0$ and Corollary~\ref{cfst}
\begin{align*}
\int_{\R}|\tau||\cF\cJ_0\cR g(v,\cdot)(\tau)|^2{\rm d}\tau&=\int_{\R}|\tau|^2|\cF\cR g(v,\cdot)(\tau)|^2{\rm d}\tau\\
&=\int_{\R}|\tau|^2|\cF g(\tau,\tau v)|^2{\rm d}\tau<+\infty
\end{align*} 
since by definition $g$ is a smooth function. We set $\zeta(v,\tau)=\cJ_0^2\mathcal{R}g(v,\tau)$, that is 
\[
\cF\zeta(v,\cdot)(\tau)=|\tau|\cF g(\tau,\tau v),
\]
so that \eqref{Wcoeff1} becomes 
\begin{align}\label{Wcoeff2}
\langle f,T_{b}g\rangle=\int_{\R}\langle\mathcal{R}f(v,\cdot),T_{n(v)\cdot b}\zeta(v,\cdot)\rangle{\rm d}v.
\end{align}
Furthermore, if possible, we choose $g$ of the form 
\[
\cF g(\xi_1,\xi_2)=\cF g_1(\xi_1)\cF g_2\left(\frac{\xi_2}{\xi_1}\right),
\]
with $g_1\in L^1(\R^2)\cap L^2(\R^2)$ satisfying the condition 
\begin{equation}\label{conditionh}
\int_{\mathbb{R}}|\tau|^{2}|\mathcal{F}g_1(\tau)|^2\ {\rm d}\tau<+\infty
\end{equation}
and $g_2\in L^1(\R^2)\cap L^2(\R^2)$. Under these hypotheses, \eqref{Wcoeff2} becomes 
\begin{align*}
\langle f,T_{b}g\rangle=\int_{\R}\langle\mathcal{R}f(v,\cdot),T_{n(v)\cdot b}\zeta_1\rangle \zeta_2(v){\rm d}v,
\end{align*}
where $\zeta_1=\cJ_o^2 g_1$, which is well-defined by \eqref{conditionh}, and $\zeta_2=\cF g_2$.
\end{proof}
Theorem~\ref{reconstructionteo1} and formulas \eqref{coeffh}, \eqref{coeffv} and \eqref{Wcoeff2} give our main result. We recall that $\psi$ is an admissible vector of the form \eqref{eq:2} satisfying conditions \eqref{eq:4} and such that $\psi_1$ satisfies \eqref{newcondition}.
Furthermore we require that $\psi$ is smooth with infinitely directional vanishing moments in the $x_1$-direction \cite{gr11}.
\begin{theorem}\label{reconstructionteo}
For any $f\in L^1(\R^2)\cap L^2(\R^2)$, we have the reconstruction formula 
\begin{align}\label{eq:10}
\nonumber\|f\|^2&=\int_{\R^2}|\langle f,T_{b}g\rangle|^2\ {\rm d}b+\int_{-1}^{1}\int_{-2}^{2}\int_{\R^2}|\mathcal{S}_{\psi}[P_{C}f](b,s,a)|^2\ {\rm d}b{\rm d}s\frac{{\rm d}a}{|a|^3}\\
&+\int_{-1}^{1}\int_{-2}^{2}\int_{\R^2}|\mathcal{S}^{\bf v}_{\psi^{\bf v}}[P_{C^{\bf v}}f](b,s,a)|^2\ {\rm d}b{\rm d}s\frac{{\rm d}a}{|a|^3},
\end{align}
where $g$ is a smooth function in $L^1(\R^2)\cap L^2(\R^2)$ such that \eqref{W} holds true and for any $b=(x,y)\in\R^2$, $s\in\R$, $a\in\R^{\times}$
\begin{align*}
&\mathcal{S}_{\psi}[P_{C}f](x,y,s,a)=|a|^{-\frac{3}{4}}\int_{-1}^{1}\mathcal{W}_{\chi_1}(\cR f(v,\cdot))(x+vy,a)\overline{\phi_2\left(\frac{v-s}{|a|^{1/2}}\right)}\ {\rm d}v,\\
&\mathcal{S}^{\bf v}_{\psi^{\bf v}}[P_{C^{\bf v}}f](x,y,s,a)=|a|^{-\frac{3}{4}}\int_{-1}^1\mathcal{W}_{\chi_1}(\cR^{\bf v}f(v,\cdot))(vx+y,a)\overline{\phi_2\left(\frac{v-s}{|a|^{1/2}}\right)}\ {\rm d}v,\\
&\langle f,T_{(x,y)}g\rangle=\int_{\R}\langle\mathcal{R}f(v,\cdot),T_{x+vy}\zeta(v,\cdot)\rangle{\rm d}v,
\end{align*}
where $\zeta=\cJ_0^2\mathcal{R}g$, $\mathcal{F}\chi_1(\tau)=|\tau|\mathcal{F}\psi_{1}(\tau)$, $\phi_2=\cF\psi_2$ .
\end{theorem}
\begin{proof}
The proof follows immediately by Theorem~\ref{reconstructionteo1} and Propositions~\ref{propositionfirst}~and~\ref{hcoeff}.
\end{proof}
This theorem gives an alternative reproducing formula for any $f\in L^1(\R^2)\cap L^2(\R^2)$ in which, by the "shearlets on the cone" construction, the scale and shearing parameters range over compact sets and, by Propositions~\ref{propositionfirst}~and~\ref{hcoeff}, the coefficients depend on $f$ only through its Radon transform. Therefore equation \eqref{eq:10} allows to reconstruct an unknown signal $f$ from its Radon transform by computing the family of coefficients 
$\{\langle f,T_{b}g\rangle,\mathcal{S}_{\psi}[P_{C}f](b,s,a),\mathcal{S}^v_{\psi^{v}}[P_{C^{\bf v}}f](b,s,a)\}_{b\in\R^2,s\in\R,a\in\R^{\times}}$ by means of Theorem~\ref{reconstructionteo}. It is worth observing that the different contributions in \eqref{eq:10} with $\cR f(v,t)$ and $\cR^{\bf v}f(v,t)$, $|v|\leq1$, reconstruct the frequency projections $P_Cf$ and $P_{C^{\bf v}}f$, respectively.
\section{Generalizations}\label{sec:final}\begin{figure}[h]
\centering
	\includegraphics[scale=.5]{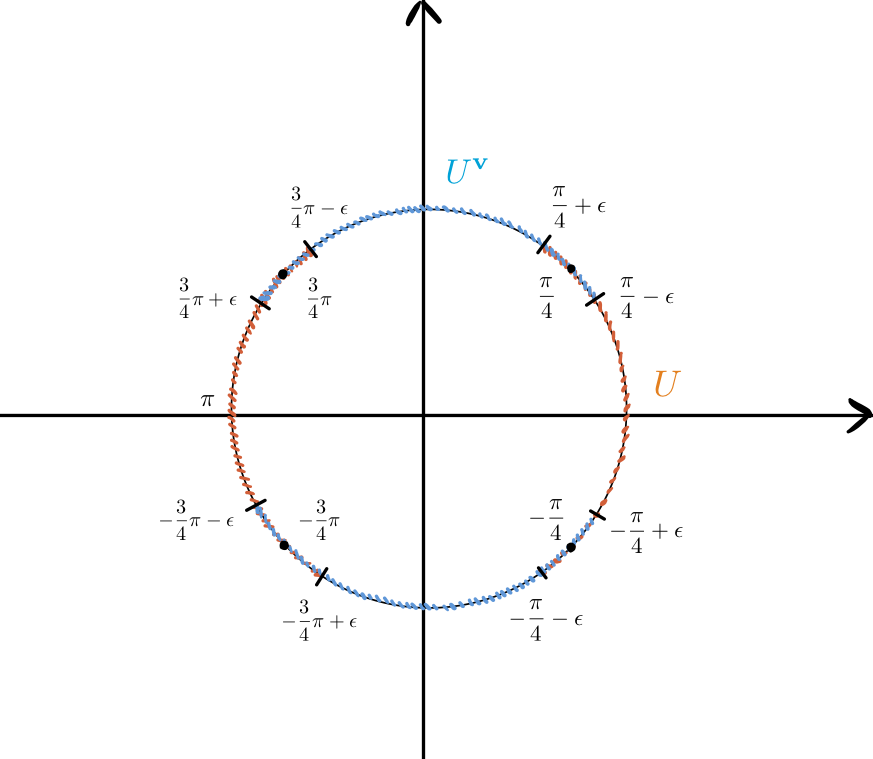}
	%
	%
	\caption{The open cover $\{U,U^{\bf v}\}$ of the unit circle in the plane $S^1\simeq (-\pi,\pi]$.}
	\label{JSfig:3}       
\end{figure} 
A disadvantage in formula \eqref{reconstructionghros}, and therefore in formula \eqref{eq:10}, is that the frequency projections $P_{C}$ and $P_{C^{\bf v}}$ performed on $f$ can lead to artificially slow decaying shearlet coefficients.
In order to avoid this problem we consider an open cover $\{U,U^{\bf v}\}$ of the unit circle in the plane $S^1\simeq (-\pi,\pi]$, where 
\[
\begin{split}
U=&(-\pi, -\frac{3}{4}\pi+\epsilon)\cup(-\frac{\pi}{4}-\epsilon,\frac{\pi}{4}+\epsilon)\cup(\frac{3}{4}\pi-\epsilon,\pi],\\
&U^{\bf v}=(-\frac{3}{4}\pi-\epsilon,-\frac{\pi}{4}+\epsilon)\cup(\frac{\pi}{4}-\epsilon,\frac{3}{4}\pi+\epsilon).
\end{split}
\]
Then, there exist even functions $\varphi,\varphi^{\bf v}\in C^{\infty}((-\pi,\pi])$ such that $\text{supp}\,\varphi\subseteq U$, $\text{supp}\,\varphi^{\bf v}\subseteq U^{\bf v}$
and $\varphi(\theta)^2+\varphi^{\bf v}(\theta)^2=1$ for all $\theta\in (-\pi,\pi]$, see \cite{lani07}. For any $\xi\in\R^2\setminus\{0\}$, we denote by $\theta_{\xi}\in(-\pi,\pi]$
the angle corresponding to $\xi/|\xi|\in S^1$ by the canonical isomorphism $S^1\simeq (-\pi,\pi]$. Then, we define the functions $\Phi,\,\Phi^{\bf v}\in C^{\infty}(\R^2\setminus\{0\})$ by
\[
\Phi(\xi)=\varphi(\theta_{\xi}),\qquad \Phi^{\bf v}(\xi)=\varphi^{\bf v}(\theta_{\xi}).
\] 
It is easy to verify that $\text{supp}\,\Phi=\left\{(\xi_1,\xi_2)\in\R^2:|\xi_2/\xi_1|\leq\tan{(\frac{\pi}{4}+\epsilon)}\right\}$, $\text{supp}\,\Phi^{\bf v}=\left\{(\xi_1,\xi_2)\in\R^2:|\xi_1/\xi_2|\leq\cot{(\frac{\pi}{4}-\epsilon)}\right\}$
and $\Phi(\xi)^2+\Phi^{\bf v}(\xi)^2=1$ for all $\xi\in\R^2\setminus\{0\}$.
We define the operators $L\colon L^2(\R^2)\to L^2(\R^2)$ and $L^{\bf v}\colon L^2(\R^2)\to L^2(\R^2)$ as follows
\begin{equation*}
\cF(Lf)(\xi)=\cF f(\xi)\Phi(\xi)
\end{equation*}
and
\begin{equation*}
\cF(L^{\bf v}f)(\xi)=\cF f(\xi)\Phi^{\bf v}(\xi).
\end{equation*}
We recall that $\psi$ is an admissible vector of the form \eqref{eq:2} satisfying conditions \eqref{eq:4} and such that $\psi_1$ satisfies \eqref{newcondition}.
Using analogous computations as in Section~\ref{sec:main}, it is possible to show that for any $f\in L^1(\R^2)\cap L^2(\R^2)$ and $(x,y,s,a)\in \R^2\times\R\times\R^{\times}$
\begin{align}\label{Th}
\nonumber\mathcal{S}_{\psi}[Lf]&(x,y,s,a)\\
&=|a|^{-\frac{3}{4}}\int_{\R}\mathcal{W}_{\chi_1}(\cR f(v,\cdot))(x+vy,a)\overline{\phi_2\left(\frac{v-s}{|a|^{1/2}}\right)}\varphi(\arctan{v})\ {\rm d}v,
\end{align}
\begin{align}\label{Tv}
\nonumber\mathcal{S}^{\bf v}_{\psi^{\bf v}}[L^{\bf v}f]&(x,y,s,a)\\
&=|a|^{-\frac{3}{4}}\int_{\R}\mathcal{W}_{\chi_1}(\cR^{\bf v}f(v,\cdot))(vx+y,a)\overline{\phi_2\left(\frac{v-s}{|a|^{1/2}}\right)}\varphi^{\bf v}\left(\arctan{\frac{1}{v}}\right)\ {\rm d}v,
\end{align}
where $\mathcal{F}\chi_1(\tau)=|\tau|\mathcal{F}\psi_{1}(\tau)$ and $\phi_2=\cF\psi_2$.
Furthermore, following the proof of Theorem 3 in \cite[Chapter 2]{kula12}, it is possible to derive a reconstruction formula of the form \eqref{eq:10} in this new setup. 
\begin{theorem}\label{reconstructionteofinal}
For any $f\in L^1(\R^2)\cap L^2(\R^2)$, we have the reconstruction formula 
\begin{align}\label{reconstructionghros2}
\nonumber\|f\|^2&=\int_{\R^2}|\langle f,T_{b}g\rangle|^2\ {\rm d}b+\int_{-1}^{1}\int_{-2}^{2}\int_{\R^2}|\mathcal{S}_{\psi}[Lf](b,s,a)|^2\ {\rm d}b{\rm d}s\frac{{\rm d}a}{|a|^3}\\
&+\int_{-1}^{1}\int_{-2}^{2}\int_{\R^2}|\mathcal{S}^{\bf v}_{\psi^{\bf v}}[L^{\bf v}f](b,s,a)|^2\ {\rm d}b{\rm d}s\frac{{\rm d}a}{|a|^3},
\end{align}
where $g$ is a smooth function in $L^1(\R^2)\cap L^2(\R^2)$ such that for all $\xi\in\R^2$
\begin{align}\label{gcondition}
\nonumber|\cF g(\xi)|^2&+\Phi(\xi)^2\int_{-1}^{1}\int_{-2}^{2}|\cF\psi(A_a{^t\!N_s}\xi)|^2\ {\rm d}s\frac{{\rm d}a}{a^{3/2}}\\
&+\Phi^{\bf v}(\xi)^2\int_{-1}^{1}\int_{-2}^{2}|\cF\psi^{\bf v}(\tilde{A_a}N_s\xi)|^2\ {\rm d}s\frac{{\rm d}a}{a^{3/2}}=1
\end{align}
and the coefficients in \eqref{reconstructionghros2} are given by \eqref{Th}, \eqref{Tv} and \eqref{Wcoeff2}.
\end{theorem}
\begin{proof}
Consider a smooth function $g\in L^1(\R^2)\cap L^2(\R^2)$ such that \eqref{gcondition} holds true.
By Plancherel theorem, we have that 
\begin{align}\label{firstequality2}
\nonumber\int_{\R^2}|\langle f,T_{b}g\rangle|^2\ {\rm d}b&=\int_{\R^2}\left|\int_{\R^2}\cF f(\xi)\overline{\cF g(\xi)}e^{2\pi ib\cdot\xi}\ {\rm d}\xi\right|^2\ {\rm d}b\\
\nonumber&=\int_{\R^2}|\cF^{-1}(\cF f\ \overline{\cF g})(b)|^2\ {\rm d}b\\
&=\int_{\R^2}|\cF f(\xi)|^2|\cF g(\xi)|^2\ {\rm d}\xi.
\end{align}
Using an analogous computation, by Plancherel theorem and Fubini's theorem we have
\begin{align}\label{secondequality2}
\nonumber&\int_{-1}^{1}\int_{-2}^{2}\int_{\R^2}|\mathcal{S}_{\psi}[Lf](b,s,a)|^2\ {\rm d}b{\rm d}s\frac{{\rm d}a}{a^3}\\
\nonumber&=\int_{-1}^{1}\int_{-2}^{2}\int_{\R^2}|\langle Lf, S_{b,s,a}\psi\rangle|^2\ {\rm d}b{\rm d}s\frac{{\rm d}a}{a^3}\\
\nonumber&=\int_{-1}^{1}\int_{-2}^{2}\int_{\R^2}\left|\int_{\R^2}\cF f(\xi)\Phi(\xi)\overline{\cF\psi(A_a{^t\!N_s}\xi)}e^{2\pi i\xi b}\ {\rm d}\xi\right|^2\ {\rm d}b{\rm d}s\frac{{\rm d}a}{a^{3/2}}\\
\nonumber&=\int_{-1}^{1}\int_{-2}^{2}\int_{\R^2}|
\cF^{-1}(\cF f\Phi\overline{\cF\psi(A_a{^t\!N_s}\cdot)})(b)|^2\ {\rm d}b{\rm d}s\frac{{\rm d}a}{a^{3/2}}\\
\nonumber&=\int_{-1}^{1}\int_{-2}^{2}\int_{\R^2}|\cF f(\xi)|^2\Phi(\xi)^2|\cF\psi(A_a{^t\!N_s}\xi)|^2\ {\rm d}\xi{\rm d}s\frac{{\rm d}a}{a^{3/2}}\\
&=\int_{\R^2}|\cF f(\xi)|^2\Phi(\xi)^2\int_{-1}^{1}\int_{-2}^{2}|\cF\psi(A_a{^t\!N_s}\xi)|^2\ {\rm d}s\frac{{\rm d}a}{a^{3/2}}{\rm d}\xi.
\end{align} 
Similarly, we have that 
\begin{align}\label{thirdquality2}
\nonumber&\int_{-1}^{1}\int_{-2}^{2}\int_{\R^2}|\mathcal{S}^{\bf v}_{\psi^{\bf v}}[L^{\bf v}f](b,s,a)|^2\ {\rm d}b{\rm d}s\frac{{\rm d}a}{a^3}\\
&=\int_{\R^2}|\cF f(\xi)|^2\Phi^{\bf v}(\xi)^2\int_{-1}^{1}\int_{-2}^{2}|\cF\psi^{\bf v}(\tilde{A_a}N_s\xi)|^2\ {\rm d}s\frac{{\rm d}a}{a^{3/2}}{\rm d}\xi.
\end{align} 
Thus, combining equations \eqref{firstequality2}, \eqref{secondequality2} and \eqref{thirdquality2} we obtain the reconstruction formula 
\begin{align}\label{eqfinal}
\nonumber\|f\|^2&=\int_{\R^2}|\langle f,T_{b}g\rangle|^2\ {\rm d}b+\int_{-1}^{1}\int_{-2}^{2}\int_{\R^2}|\mathcal{S}_{\psi}[Lf](b,s,a)|^2\ {\rm d}b{\rm d}s\frac{{\rm d}a}{|a|^3}\\
&+\int_{-1}^{1}\int_{-2}^{2}\int_{\R^2}|\mathcal{S}^{\bf v}_{\psi^{\bf v}}[L^{\bf v}f](b,s,a)|^2\ {\rm d}b{\rm d}s\frac{{\rm d}a}{|a|^3},
\end{align}
for any $f\in L^1(\R^2)\cap L^2(\R^2)$, where the shearlet coefficients are given by \eqref{Th} and \eqref{Tv}. It is worth observing that there always exists a function $g$
satisfying~\eqref{gcondition} provided that the admissible vector $\psi$ is smooth and possesses infinitely vanishing moments in the $x_1$-direction \cite{gr11}. Indeed, we have that 
\[
\begin{split}
z(\xi)&:=1 -\Phi(\xi)^2\int_{-1}^{1}\int_{-2}^{2}|\cF\psi(A_a{^t\!N_s}\xi)|^2\ {\rm d}s\frac{{\rm d}a}{a^{3/2}}\\
&-\Phi^{\bf v}(\xi)^2\int_{-1}^{1}\int_{-2}^{2}|\cF\psi^{\bf v}(\tilde{A_a}N_s\xi)|^2\ {\rm d}s\frac{{\rm d}a}{a^{3/2}}\\
&=\Phi(\xi)^2(1 -\int_{-1}^{1}\int_{-2}^{2}|\cF\psi(A_a{^t\!N_s}\xi)|^2\ {\rm d}s\frac{{\rm d}a}{a^{3/2}})\\
&+\Phi^{\bf v}(\xi)^2(1-\int_{-1}^{1}\int_{-2}^{2}|\cF\psi^{\bf v}(\tilde{A_a}N_s\xi)|^2\ {\rm d}s\frac{{\rm d}a}{a^{3/2}})\\
\end{split}
\]
Following the proof of Lemma 3 in \cite[Chapter 2]{kula12} it is possible to prove that 
\[
\begin{split}
1 -\int_{-1}^{1}\int_{-2}^{2}|\cF\psi(A_a{^t\!N_s}\xi)|^2\ {\rm d}s\frac{{\rm d}a}{a^{3/2}}=O(|\xi|^{-N}),\,\,\,\left|\frac{\xi_2}{\xi_1}\right|\leq\tan{(\frac{\pi}{4}+\epsilon)},
\end{split}
\]
for all $N\in\N$.
Analogously, 
\[
\begin{split}
1-\int_{-1}^{1}\int_{-2}^{2}|\cF\psi^{\bf v}(\tilde{A_a}N_s\xi)|^2\ {\rm d}s\frac{{\rm d}a}{a^{3/2}}=O(|\xi|^{-N}),\,\,\,\left|\frac{\xi_1}{\xi_2}\right|\leq\cot{(\frac{\pi}{4}-\epsilon)},
\end{split}
\]
for all $N\in\N$.
 Therefore, there exists a
smooth function $g\in L^1(\R^2)\cap L^2(\R^2)$ such that $\cF g(\xi)=\sqrt{z(\xi)}$, so
that~\eqref{gcondition} holds true. Finally, by Proposition~\ref{hcoeff}, we can express the coefficients $\langle f,T_{b}g\rangle$ in reconstruction formula \eqref{eqfinal} in terms of $\cR f$ only.
\end{proof}

\section{Acknowlegments}
\label{JSsec:9}

F. Bartolucci, F. De Mari and E. De Vito are members of the Gruppo Nazionale per
l'Analisi Matematica, la Probabilit\`a e le loro Applicazioni (GNAMPA)
of the Istituto Nazionale di Alta Matematica (INdAM).




\bibliography{biblio2.bib}

\end{document}